\documentclass{amsart}
\usepackage{amssymb}
\usepackage{needspace}
\usepackage[hidelinks]{hyperref}
\numberwithin{equation}{section}
\newtheorem{theorem}{Theorem}[section]
\newtheorem{proposition}[theorem]{Proposition}
\newtheorem{lemma}[theorem]{Lemma}
\newtheorem{corollary}[theorem]{Corollary}
\newtheorem{conjecture}[theorem]{Conjecture}
\theoremstyle{definition}
\newtheorem{definition}[theorem]{Definition}
\theoremstyle{remark}
\newtheorem{remark}[theorem]{Remark}
\newcommand{\Q}{\mathbf Q}
\newcommand{\Z}{\mathbf Z}
\newcommand{\F}{\mathbf F}

\newcommand{\GL}{\operatorname{GL}}
\newcommand{\SL}{\operatorname{SL}}
\newcommand{\PGL}{\operatorname{PGL}}
\newcommand{\Gal}{\operatorname{Gal}}
\newcommand{\Cl}{\operatorname{Cl}}
\newcommand{\tr}{\operatorname{tr}}
\newcommand{\detm}{\operatorname{det}}

\newcommand{\ssem}{\mathrm{ss}}
\newcommand{\ab}{\mathrm{ab}}

\newcommand{\ps}{\mathrm{ps}}

\newcommand{\GG}{\mathcal G}

\title[Unramified Fontaine-Mazur Conjecture for $\operatorname{GL}_2$]
{On Boston's Unramified Conjecture for $\operatorname{GL}_2$ and McLeman's $(3,3)$-Conjecture}
\author{Yufan Luo}
\address{Shanghai Institute for Mathematics and Interdisciplinary Sciences (SIMIS), Shanghai 200433, China}
\address{Research Institute of Intelligent Complex Systems, Fudan University, Shanghai 200433, China}
\email{yufanluo@hotmail.com}
\subjclass[2020]{Primary 11F80; Secondary 11R37}
\keywords{Fontaine--Mazur conjecture, Galois representations, class field towers, conjugate self-dual representations, McLeman's conjecture}

\begin{document}
	\begin{abstract}
        Let $p$ be an odd prime and let $A$ be a complete Noetherian commutative local ring with finite residue field of characteristic $p$. Let $G_{\Q}$ be the absolute Galois group of the rational number field and fix an inertia subgroup $I_p\subset G_\Q$ at $p$. We prove that every continuous odd representation $\rho:G_\Q\to\GL_2(A)$, unramified outside finitely many primes and with finite $\rho(I_p)$, has finite image if $p>3$ or $A$ has positive characteristic. For $p=3$, the conclusion also holds when the local representation at $3$ is reducible, or the local residual semisimplification has no mod-$3$ cyclotomic self-twist. In particular, this proves the two-dimensional odd case of Boston's unramified Fontaine--Mazur conjecture over $\Q$ for every odd $p$. Furthermore, we deduce finiteness results for conjugate self-dual representations of pro-$p$ Galois groups over imaginary quadratic fields under a ramification condition, and for representations of their unramified pro-$p$ groups of generator rank at most two.
		As an application, we resolve the sufficiency direction of McLeman's $(3,3)$-conjecture on $p$-class field towers for $p>3$. Namely, we prove that if $K$ is an imaginary quadratic field with $p$-class rank two and the Galois group $\operatorname{Gal}(K_{\varnothing}(p)/K)$ of the maximal unramified $p$-extension of $K$ has Zassenhaus type $(3,3)$, then the $p$-class field tower of $K$ is finite. For $p=3$, we also rule out infinite towers for ten of the thirteen fourth Zassenhaus quotients classified by Ahlqvist and Pink.
	\end{abstract}

	\maketitle
	\tableofcontents
	
	\section{Introduction}
	
	\subsection{Odd representations of $G_{\Q}$}
	Let $p$ be a prime number. By a \emph{pro-$p$ ring} we mean a
	complete Noetherian commutative local ring with finite residue
	field of characteristic $p$. For a number field $K$, let $G_K$
	denote its absolute Galois group and let $S_\infty(K)$ denote the
	set of infinite primes of $K$. For a finite set
	$S$ of primes of $K$, let $K_S$
	(resp. $K_S(p)$) be the maximal algebraic extension
	(resp. $p$-extension) of $K$ unramified outside $S$. We put
	$G_{K,S}:=\Gal(K_S/K)$ and $G_{K,S}(p):=\Gal(K_S(p)/K)$.
	For each rational prime $\ell$, choose a decomposition subgroup
	$G_{\Q_\ell}\subset G_\Q$ and write $I_\ell\subset G_{\Q_\ell}$
	for its inertia subgroup. Boston's strengthening \cite[Conjecture~2]{Boston99} of the unramified
	Fontaine--Mazur conjecture \cite[Conjecture~5a]{FM95} is the following.
	
	\begin{conjecture}\label{conj:boston}
		Let $K$ be a number field, let $S$ be a finite set of primes of $K$
		containing $S_\infty(K)$ and no prime above $p$, and let
		$n\geq1$. Then every continuous homomorphism
		\[
		G_{K,S}\longrightarrow\GL_n(A)
		\]
		has finite image for every pro-$p$ ring $A$.
	\end{conjecture}
	
	For $n=2$ and $p>2$, Allen and Calegari
	\cite[Corollary~3]{AllenCalegari14} proved
	Conjecture~\ref{conj:boston} over totally real number fields $K$ for totally odd representations satisfying some residual-image
	hypotheses. Over $\Q$, Calegari and Geraghty
	\cite[Corollary~1.6]{CalegariGeraghty18} proved it for minimal lifts
	of odd irreducible residual representations. Also, Calegari
	\cite[Theorem~1.1]{Calegari18} proved it for odd $p$-adic representations with absolutely irreducible reduction that are locally reducible at every ramified prime. See also \cite[Remark 1.2.3]{Zhang26}.
	
     We say that a continuous representation
    $\rho:G_\Q\to\GL_2(A)$ is \emph{odd} if $\detm\rho(c)=-1$ for a complex conjugation $c$. We will write $\bar\rho$ for reduction of $\rho$ modulo the maximal ideal of $A$. Let $\omega:G_{\Q_p}\to\F_p^\times$ be the mod-$p$ cyclotomic character. 
    
    Our first main result is the following.

    \begin{theorem}\label{thm:finite-inertia-intro}
        Let $p>2$, let $A$ be a pro-$p$ ring, and let $\rho:G_\Q\to\GL_2(A)$ be a continuous odd representation
        unramified outside finitely many primes. Suppose that
        $\rho(I_p)$ is finite. Then $\rho$ has finite image if at
        least one of the following conditions holds:
        \begin{enumerate}
        \renewcommand{\labelenumi}{\textup{(\roman{enumi})}}
        \item $p>3$;
        \item $A$ has positive characteristic;
        \item $p=3$ and $\rho|_{G_{\Q_3}}$ is reducible over $A$;
        \item $p=3$ and, writing
        $\tau=(\bar\rho|_{G_{\Q_3}})^{\ssem}$, one has
        $\tau\not\simeq\tau\otimes\omega$, where $(\bar\rho|_{G_{\Q_3}})^{\ssem}$ denotes
        the semisimplification of $\bar\rho|_{G_{\Q_3}}$ over an algebraic closure of the residue field of $A$.
        \end{enumerate}
    \end{theorem}
   
   Our proof relies on the recent work \cite{Zhang26} of Zhang. The proof first treats representations over local fields. In characteristic zero, the weight-one theorems of Zhang and Pilloni--Stroh handles \textup{(i)}, \textup{(iii)}, and~\textup{(iv)}. In the locally reducible case in characteristic $p$, the finiteness of the ordinary deformation
    and pseudodeformation rings (Theorem~\ref{thm:pan-ordinary}) implies that
    the global traces take only finitely many values. In the locally absolutely irreducible case in characteristic $p$, the finiteness of $\rho(I_p)$
    and the finite order of $\detm\rho$ imply that $\rho(G_{\Q_p})$ is finite, and hence Zhang's
    local-to-global finiteness theorem \cite[Theorem~1.2.9]{Zhang26} gives the same conclusion for the global traces. We then pass to general pro-$p$ ring $A$ by the specialization argument in \cite[proof of Theorem~3.9]{Luo26}.

 We have the following immediate consequence for every odd $p$.

    \begin{corollary}\label{thm:combined-intro}
        Let $p>2$ and let $S$ be a finite set of primes
        of $\Q$ containing $\infty$ but not $p$. Then every continuous odd
        representation $G_{\Q,S}\to\GL_2(A)$ has finite image, where
        $A$ is a pro-$p$ ring.
    \end{corollary}

    This  proves Conjecture~\ref{conj:boston} for $K=\Q$, $n=2$, and
    odd representations for every $p>2$ without any residual hypothesis. Moreover, we also have the following corollary.
    
    \begin{corollary}\label{thm:descent-intro}
        Let $p>2$, let $K/\Q$ be a finite Galois extension, and let
        $S$ be a finite $\Gal(K/\Q)$-stable set of primes of $K$
        containing $S_\infty(K)$ and no prime above $p$.
        Let $A$ be a pro-$p$ ring and let
        \[
        \rho:\Gal(K_S/\Q)\longrightarrow\GL_2(A)
        \]
        be a continuous representation. We also use $\rho$ for its inflation to $G_{\Q}$. Suppose that $\rho$ is odd. Then $\rho$ has finite image if \emph{at least one} of the following conditions holds:
        \begin{enumerate}
        \renewcommand{\labelenumi}{\textup{(\roman{enumi})}}
        \item $p>3$;
        \item $A$ has positive characteristic;
        \item $p=3$ and $\rho|_{G_{\Q_3}}$ is reducible over $A$;
        \item $p=3$ and, writing
        $\tau=(\bar\rho|_{G_{\Q_3}})^{\ssem}$, one has $  \tau\not\simeq\tau\otimes\omega$;
        \item $p=3$ and $K_v/\Q_3$ is abelian for every $v\mid3$;
        \item $p=3$, $K/\Q$ is tamely ramified at $3$, and
        $\#\rho(I_3)\ne4$.
        \end{enumerate}
    \end{corollary}

	\subsection{Representations of pro-$p$ Galois groups over imaginary quadratic fields}
	Having established the odd two-dimensional case of Boston's conjecture over $\Q$ in Section~1.1, we now prove some even cases over imaginary quadratic fields. Let $K/\Q$ be imaginary quadratic, write $\operatorname{disc}(K)$ for the field discriminant of $K$, $\Cl(K)$ for the ideal class group of $K$, and $d_p\Cl(K)=\dim_{\F_p}\Cl(K)/p\Cl(K)$. Fix a complex conjugation $c\in G_\Q$. For a finite set $S$ of primes of $K$ containing $S_\infty(K)$ and no prime above $p$, put
	\[
	T_K(S)=\{\ell\text{ rational prime}:\text{ every prime of }K\text{ above }\ell\text{ belongs to }S\}.
	\]
	For $\rho:G_{K,S}(p)\to\GL_2(A)$, let $\rho_K$ be its inflation to $G_K$. We call $\rho$ \emph{conjugate self-dual} if $\rho_K^c\simeq\rho_K^\vee$ over $A$, where
	\[
	\rho_K^c(g)=\rho_K(cgc^{-1}),\qquad \rho_K^\vee(g)={}^t\!\bigl(\rho_K(g)^{-1}\bigr).
	\]
	Here ${}^t M$ denotes the transpose of the matrix $M$. Our second main results are the following.

	\begin{theorem}\label{thm:csd-intro}
		Let $p>2$ be a prime, let $K/\Q$ be imaginary quadratic, and let $S$ be as above. Let $A$ be a pro-$p$ ring. If
		\[
		\#\{\ell\in T_K(S):\ell\equiv1\pmod p\}\leq1,
		\]
		then every continuous conjugate self-dual representation $\rho:G_{K,S}(p)\to\GL_2(A)$ has finite image.
	\end{theorem}

	\begin{theorem}\label{thm:rank2-intro}
		Let $p>2$ be a prime and let $K/\Q$ be imaginary quadratic with $d_p\Cl(K)\leq2$. Let $A$ be a pro-$p$ ring. Then every continuous representation $G_{K,\varnothing}(p)\to\GL_2(A)$ has finite image.
	\end{theorem}

    Note that the group $G_{K,\varnothing}(p)$ in Theorem~\ref{thm:rank2-intro} can itself be infinite. For example, the imaginary quadratic field of discriminant $-2724783$
    has $5$-class rank two and an infinite $5$-class field tower; see \cite[Theorem 5.8]{AC25}. The proofs use conjugate self-duality to extend a finite-order twist of a representation to an odd representation of $G_{\mathbf{Q}}$, to which Corollary~\ref{thm:descent-intro} applies. For the tower group $G_{K,\varnothing}(p)$ with two generators, an elementary matrix argument supplies the required involution.

	\begin{theorem}\label{thm:faithful-tower}
		Let $p>2$ be a prime and let $K/\Q$ be imaginary quadratic. If the group $G_{K,\varnothing}(p)$ admits a faithful continuous homomorphism into $\GL_2(A)$ for some pro-$p$ ring $A$, then $G_{K,\varnothing}(p)$ is finite. 
	\end{theorem}
	For $d_p\Cl(K)\leq2$, this follows from Theorem~\ref{thm:rank2-intro}. For larger rank, Koch--Venkov's theorem~\cite{KV74} makes the pro-$p$ group $G_{K,\varnothing}(p)$ Golod--Shafarevich. Zelmanov~\cite{Zelmanov00} showed that such a group contains a closed nonabelian free pro-$p$ subgroup,
    and Zubkov~\cite{Zubkov87} showed that no such subgroup embeds continuously in $\GL_2(A)$. 
	
	\subsection{McLeman's \texorpdfstring{$(3,3)$}{(3,3)}-conjecture}
	Recall that the extension $K_{\emptyset}(p)$ is the union of the fields obtained by
	iterating the Hilbert $p$-class field construction.  Its finiteness
	is therefore equivalent to the termination of the $p$-class field
	tower of $K$.  A decisive development in the class field tower
	problem was the theorem of Golod and Shafarevich \cite{GS64}, which
	established the existence of infinite towers by relating the numbers
	of generators and relations of their Galois groups.  By class field
	theory, the minimal number of generators of $G_{K,\varnothing}(p)$
	is $d=d_p\Cl(K)$.  If $d=0$, then $G_{K,\varnothing}(p)$ is trivial;
	if $d=1$, it is finite cyclic.  For larger $d$, the arithmetic of
	imaginary quadratic fields imposes additional restrictions on a
	minimal pro-$p$ presentation.  Shafarevich's relation rank theorem
	gives exactly $d$ defining relations \cite{Sha66}; Koch and Venkov
	show that, for odd $p$, the action of complex conjugation forces their
	Zassenhaus depths to be odd and at least three \cite{KV74}.  Combined
	with the Golod--Shafarevich inequality, these restrictions imply
	that $G_{K,\varnothing}(p)$ is infinite whenever $d\geq3$.
	Thus $d=2$ is the only generator rank for which these general
	results do not decide whether the tower terminates.
	
	In the two-generator case, the results of Golod--Shafarevich and
	Koch--Venkov restrict finite tower groups to three Zassenhaus types:
	$(3,3)$, $(3,5)$, and $(3,7)$; see \cite{KV74} and
	\cite[Section~2]{McL08}.  Determining finiteness within these three
	types requires further information.  In
	\cite[Conjecture~2.9]{McL08}, McLeman proposed the following
	characterization of finite towers in the two-generator setting.
	
	\begin{conjecture}\label{conj:33}
		Let $p>2$ and let $K/\Q$ be imaginary quadratic with
		$d_p\Cl(K)=2$.  Then $K_{\emptyset}(p)/K$ is finite if and only if
		$G_{K,\varnothing}(p)$ has Zassenhaus type $(3,3)$.
	\end{conjecture}
	
	Ahlqvist and Carlson disproved the necessity direction of
	Conjecture~\ref{conj:33} in \cite[Theorem~5.7]{AC25}.  Their examples
	include the pairs
	\[
	(p,\operatorname{disc}(K))=(5,-90868),\qquad(7,-159592),
	\]
	for which the $p$-class field towers are finite although the
	Zassenhaus type differs from $(3,3)$. In this paper, we study the sufficiency implication of Conjecture~\ref{conj:33}.  
	
	\begin{theorem}\label{thm:mcleman-pgt3}
		Let $p>3$ and let $K/\Q$ be imaginary quadratic with
		$d_p\Cl(K)=2$.  If $G_{K,\varnothing}(p)$ has Zassenhaus type $(3,3)$, then
		$K_{\emptyset}(p)/K$ is finite.
	\end{theorem}
	
	This theorem establishes the sufficiency direction of McLeman's $(3,3)$-conjecture for $p>3$. For $p=3$, we also have the following theorem.
	
	\begin{theorem}\label{thm:mcleman-p3}
		Let $K/\Q$ be imaginary quadratic with $d_3\Cl(K)=2$. Suppose that $G_{K,\emptyset}(3)$ has Zassenhaus type $(3,3)$ and
		\[
		G_{K,\emptyset}(3)/D_4(G_{K,\emptyset}(3))\simeq\operatorname{SmallGroup}(243,i),
		\]
		for some $i\in\{2,4,5,6,7,8,14,15,17,18\}$, where $D_4(G)$ denotes the fourth term of the Zassenhaus filtration of $G$, and $\operatorname{SmallGroup}(243,i)$ denotes the group with that identifier in the GAP Small Groups library. Then the extension $K_{\emptyset}(3)/K$ is finite.
	\end{theorem}
	
	The group $G_{K,\varnothing}(p)$ is a strong Schur $\sigma$-group in the sense of \cite[Section 4]{Pink25}, permitting the structural
results of Pink and Ahlqvist--Pink used below. The connection with Boston's Conjecture~\ref{conj:boston} is through
linear representations of the tower group.  The group $G_{K,\varnothing}(p)$ is a
pro-$p$ quotient of $G_{K,\varnothing}$, and
Theorem~\ref{thm:rank2-intro} gives finite image for all of its
two-dimensional $p$-adic representations when $d_p\Cl(K)=2$. Note that finite image of representations alone does not force $G_{K,\varnothing}(p)$ to be finite, and one needs a faithful representation.
For $p>3$, Pink's structural theorem
\cite[Proposition~8.1]{Pink25} supplies a faithful projective
two-dimensional $p$-adic representation if a strong Schur
$\sigma$-group of type $(3,3)$ is infinite. Moreover, one can show that a compact pro-$p$ projective image lifts, after extending scalars, to $\SL_2$ because $p$ is odd. Theorem~\ref{thm:rank2-intro} then rules out the
infinite group.  The required $\GL_2$-embedding for $p>3$ can also be obtained from Andozhskii's paper \cite{Andozhskii75}.  Ahlqvist--Pink's projective embeddings in \cite[Proposition~9.2(a)]{AP26} give the analogous argument for
the ten listed fourth-quotient types at $p=3$. The cases with fourth-quotient identifiers $3,9,13$ are still open.
	
	\subsection{Organization and notation}
	Section~\ref{sec:inputs} collects the arithmetic inputs, including the weight-one theorems and the finiteness results for deformation and pseudodeformation rings.
	Section~\ref{sec:odd-completion} proves Theorem~\ref{thm:finite-inertia-intro} by passing from local coefficient fields to general pro-$p$ rings and deduces Corollaries~\ref{thm:combined-intro} and~\ref{thm:descent-intro}.
	Section~\ref{sec:quadratic} treats conjugate self-dual representations and two-generator tower groups over imaginary quadratic fields (Theorems~\ref{thm:csd-intro} and~\ref{thm:rank2-intro}) and shows that $G_{K,\varnothing}(p)$ is finite if it embeds continuously in $\GL_2(A)$ for some pro-$p$ ring $A$ (Theorem~\ref{thm:faithful-tower}).
	Section~\ref{sec:mcleman} applies the result on two-generator tower groups to prove the sufficiency direction of McLeman's $(3,3)$-conjecture for $p>3$ and, when $p=3$, for ten of the thirteen possible fourth Zassenhaus quotients.
	
	Throughout this paper, we write $\Q$ for the field of rational numbers, $\F_p$ for the field with $p$ elements, and $\Q_p$ for the field of $p$-adic numbers. Let $\Z_p$ be the ring of integers in $\Q_p$. We fix an algebraic closure $\overline{\Q}_p$ of $\Q_p$, endowed with
	its usual $p$-adic topology. For each finite prime $v$ of a number field $K$, we fix a decomposition subgroup $G_{K_v}\subset G_K$ and denote its inertia subgroup by $I_v$.

\section{Arithmetic inputs}
\label{sec:inputs}

\begin{lemma}\label{lem:arithmetic-finiteness}
Let $K$ be a number field and $S$ a finite set of primes of $K$. Let $A$ be a pro-$p$ ring, and let
$\rho:G_{K,S}\to\GL_n(A)$ be a continuous homomorphism.
Then the image of $\rho$ is topologically finitely generated.
Furthermore, if $\rho(I_v)$ is finite for every prime $v\mid p$ of $K$
and the image $\text{Im}(\rho)$ of $\rho$ contains a closed solvable
subgroup of finite index, then $\text{Im}(\rho)$ is finite.
\end{lemma}
\begin{proof}
It follows from \cite[Theorems~2.3 and~2.4]{Luo26}, applying the latter
after restriction to a finite extension of $K$ over which $\rho$ is
unramified at every prime above $p$.
\end{proof}

\Needspace{7\baselineskip}
\begin{lemma}\label{lem:finite-traces}
Let $K$ be a number field, let $S$ be a finite set of primes of $K$,
and let $F$ be a non-Archimedean local field of residue characteristic $p$.
Let $n\geq1$, and assume that either $\operatorname{char}F=0$ or $p>n$.
Let $\rho:G_{K,S}\to\GL_n(F)$ be a continuous representation.
If $\rho(I_v)$ is finite for every prime $v\mid p$ of $K$ and the set
\[
 \mathcal T=\{\tr\rho(g):g\in G_{K,S}\}
\]
is finite, then $\rho$ has finite image.
\end{lemma}
\begin{proof}
Since $1,\ldots,n$ are invertible in $F$, Newton's identities determine
the characteristic polynomial of $\rho(g)$ from the traces
$\tr\rho(g),\ldots,\tr\rho(g^n)$. Thus only finitely many
characteristic polynomials occur, and their roots form a finite set.
If $\lambda$ is an eigenvalue of $\rho(g)$, then every positive power of
$\lambda$ belongs to this set, and hence $\lambda$ is a root of unity.
Put $H=\rho(G_{K,S})$. Since the groups $\rho(I_v)$ for $v\mid p$ are finite,
there is an open normal subgroup $U$ of $H$ with $U\cap\rho(I_v)=\{1\}$
for every $v\mid p$. Let $K'/K$ be the finite Galois extension cut out by
$\rho^{-1}(U)$. The restriction of $\rho$ to $G_{K'}$ is unramified
at every prime above $p$, and its eigenvalues are still roots of unity.
Thus $U=\rho(G_{K'})$ is finite by \cite[Proposition~3.1.2]{Luo23Thesis}.
Since $[H:U]$ is finite, so is $H$.
\end{proof}

\begin{lemma}\label{lem:p3-residual}
Let $k$ be a finite field of characteristic $3$, let
$\bar\rho:G_\Q\to\GL_2(k)$ be a continuous representation, and put
$\tau=(\bar\rho|_{G_{\Q_3}})^{\ssem}$ over $\overline k$.
Let $\omega$ be the mod-$3$ cyclotomic character of $G_{\Q_3}$.
\begin{enumerate}
\renewcommand{\labelenumi}{\textup{(\roman{enumi})}}
\item If $\bar\rho$ is absolutely irreducible and
$\tau\not\simeq\tau\otimes\omega$, then
$\bar\rho|_{G_{\Q(\zeta_3)}}$ is absolutely irreducible.
\item If $\bar\rho^{\ssem}=\bar\chi_1\oplus\bar\chi_2$, then
$\tau\not\simeq\tau\otimes\omega$ if and only if
$(\bar\chi_1/\bar\chi_2)|_{G_{\Q_3}}\ne\omega$.
\end{enumerate}
\end{lemma}
\begin{proof}
For~\textup{(i)}, put $H=G_{\Q(\zeta_3)}$ and
suppose that $\bar\rho|_H$ is not absolutely irreducible. Choose an
$H$-invariant line $L\subset\overline{k}^{\,2}$ and $g\in G_\Q\setminus H$.
Absolute irreducibility of $\bar\rho$
implies $\bar\rho(g)L\ne L$. Since $H$ is normal and $g^2\in H$,
the two lines $L$ and $\bar\rho(g)L$ are preserved by $H$ and
exchanged by $g$. The operator acting as $1$ and $-1$ on these
lines intertwines $\bar\rho$ with its twist by the quadratic character
of $\Q(\zeta_3)/\Q$. This character restricts to $\omega$ on $G_{\Q_3}$,
so semisimplification gives $\tau\simeq\tau\otimes\omega$, a contradiction.

For~\textup{(ii)}, twisting by the nontrivial quadratic character
$\omega$ fixes neither local character individually. Thus it
preserves their unordered pair exactly when it exchanges them,
which is equivalent to their ratio being $\omega$.
\end{proof}

We use the following consequence of the weight-one
theorems of Zhang and Pilloni--Stroh.

\begin{theorem}\label{thm:finite-inertia-char-zero}
Let $p$ be an odd prime and let
$\rho:G_\Q\to\GL_2(\overline{\Q}_p)$ be a continuous odd irreducible
representation unramified outside a finite set of primes of $\Q$.
Suppose that $\rho(I_p)$ is finite.
\begin{enumerate}
\renewcommand{\labelenumi}{\textup{(\roman{enumi})}}
\item If $p\geq5$, then $\rho$ has finite image.
\item If $p=3$, put $\tau=(\bar\rho|_{G_{\Q_3}})^{\ssem}$ and
let $\omega$ be the mod-$3$ cyclotomic character of $G_{\Q_3}$. If
$\rho|_{G_{\Q_3}}$ is reducible over $\overline{\Q}_3$ or
$\tau\not\simeq\tau\otimes\omega$, then $\rho$ has finite image.
\end{enumerate}
\end{theorem}
\begin{proof}
	Since $\rho(I_p)$ is finite, $\rho|_{G_{\Q_p}}$ potentially unramified, and hence de Rham with Hodge--Tate weights $\{0,0\}$. For $p\geq5$, the claim \textup{(i)} follows from Zhang's weight-one theorem \cite[Theorem~1.2.2]{Zhang26}.

Suppose that $p=3$. If $\bar\rho^{\ssem}=\bar\chi_1\oplus\bar\chi_2$, then the two
characters are distinct by oddness. When $\tau\not\simeq\tau\otimes\omega$,
Lemma~\ref{lem:p3-residual}\textup{(ii)} gives
$(\bar\chi_1/\bar\chi_2)|_{G_{\Q_3}}\ne\omega$.
Thus either hypothesis in~\textup{(ii)} satisfies the condition
of \cite[Theorem~5.4.1]{Zhang26}, which shows that $\rho$ has finite image.

Now suppose that $\bar\rho$ is absolutely irreducible. If
$\rho|_{G_{\Q_3}}$ is reducible, then \cite[Theorem~A.5.3]{Zhang26}
applies. If $\tau\not\simeq\tau\otimes\omega$,
Lemma~\ref{lem:p3-residual}\textup{(i)} shows that
$\bar\rho|_{G_{\Q(\zeta_3)}}$ is absolutely irreducible. By Pilloni--Stroh's weight-one theorem \cite[Theorem~0.2 and Corollary~2.2.3]{PilloniStroh16}, $\rho$ has finite image. This completes the proof.
\end{proof}

For the positive-characteristic argument, we recall two finiteness inputs. Suppose that $\Sigma$ is a finite set
of primes containing $p>2$ and $\infty$, $k$ is a finite field of characteristic $p$, and
$\mathcal O=W(k)$ where $W(k)$ denotes the ring of infinite Witt vectors over $k$. Given an odd two-dimensional residual pseudorepresentation
$\bar T:G_{\Q,\Sigma}\to k$, fix a finite-order character
$\chi:G_{\Q,\Sigma}\to\mathcal O^\times$ lifting $\det\bar T$.

\begin{theorem}\label{thm:local-global-finite}
With the notation as above, let $R_\Q^{\ps,\chi}$ be the universal pseudodeformation ring of $\bar T$
with determinant $\chi$, and let $R_p^{\ps,\chi}$ be the corresponding
local ring for $\bar T|_{G_{\Q_p}}$ with determinant $\chi|_{G_{\Q_p}}$.
Then the natural map $R_p^{\ps,\chi}\longrightarrow R_\Q^{\ps,\chi}$ is finite.
\end{theorem}
\begin{proof}
Since $\chi$ has finite order, it is a de Rham lift of $\det\bar T$. Then the assertion follows from \cite[Theorem~1.2.9]{Zhang26}.
\end{proof}

Let $B$ be an Artinian local
$\mathcal O$-algebra with maximal ideal $\mathfrak m_B$ and residue field $k$. Recall that a deformation of $\bar\rho$ to $B$ with determinant $\chi$ is a continuous
lift $r:G_{\Q,\Sigma}\to\GL_2(B)$ with $\detm r=\chi$, up to conjugation by a matrix
congruent to $I$ modulo $\mathfrak m_B$. A pseudodeformation of $\bar T$
to $B$ with determinant $\chi$ is a continuous function $T:G_{\Q,\Sigma}\to B$
reducing to $\bar T$ and satisfying $T(1)=2$, $T(gh)=T(hg)$, and
$T(g)T(h)=T(gh)+\chi(h)T(gh^{-1})$ for all $g,h\in G_{\Q,\Sigma}$.
Here $\chi$ takes values in $B^\times$ via the structure map $\mathcal O\to B$. All characters below are continuous.

Put $U_1=1+p\Z_p$ and $\Lambda=\mathcal O[[U_1]]$. For $u\in U_1$, write
$[u]$ for its image in $\Lambda$. Since $1+p$ topologically generates $U_1$,
we have $\Lambda\simeq\mathcal O[[X]]$ with $X=[1+p]-1$ and augmentation
ideal $\mathfrak a=(X)$.

Let $C$ be a complete local $\mathcal O$-algebra with residue field $k$.
For a character $\psi:G_{\Q_p}\to C^\times$, write $\psi^{\ab}$ for the
induced character on the abelianization $G_{\Q_p}^{\ab}$. The associated \emph{weight map} is given by
\[
 w_\psi:\Lambda\longrightarrow C,\qquad
 [u]\longmapsto\psi^{\ab}(\operatorname{rec}_p(u)),
\]
where $\operatorname{rec}_p:\Q_p^\times\to G_{\Q_p}^{\ab}$ is local reciprocity.

\begin{theorem}\label{thm:pan-ordinary}
We keep the notation as above.
\begin{enumerate}
\renewcommand{\labelenumi}{\textup{(\roman{enumi})}}
\item Suppose that $\bar T=1+\bar\chi$ with $\bar\chi$ odd and
$\bar\chi|_{G_{\Q_p}}\ne1$. Let $R^{\mathrm{ord}}$ be the universal ring representing pseudodeformations $T$ with
determinant $\chi$ for which $T|_{G_{\Q_p}}=\psi_1+\psi_2$ is a sum of two
characters lifting $1$ and $\bar\chi|_{G_{\Q_p}}$, respectively.
Then $R^{\mathrm{ord}}$ is a finite $\Lambda$-module via the weight map
of the universal character $\psi_1^{\mathrm{univ}}$.
\item Suppose that $\bar T=1+\bar\chi$ with $\bar\chi$ odd and
$\bar\chi|_{G_{\Q_p}}=1$. Let $R^{\mathrm{ord}}$ be the universal ring representing pairs $(T,\psi_1)$, where $T$ is
a pseudodeformation with determinant $\chi$ and $\psi_1:G_{\Q_p}\to B^\times$
is a chosen character lifting $1$. Put $\psi_2=\chi|_{G_{\Q_p}}\psi_1^{-1}$.
We require $T|_{G_{\Q_p}}=\psi_1+\psi_2$ and
\begin{equation}\label{eq:pan-ordinary-identity}
\begin{split}
T(\sigma\tau h)-\psi_1(\sigma)T(\tau h)-\psi_2(\tau)T(\sigma h)
+\psi_1(\sigma)\psi_2(\tau)T(h)=0
\end{split}
\end{equation}
for $\sigma,\tau\in G_{\Q_p}$ and $h\in G_{\Q,\Sigma}$.
Then $R^{\mathrm{ord}}$ is a finite $\Lambda$-module via the weight map
of the universal character $\psi_1^{\mathrm{univ}}$.
\item Suppose that $\bar T=\tr\bar\rho$ for an odd absolutely irreducible
$\bar\rho$ with $\bar\rho|_{G_{\Q_p}}$ reducible. After enlarging $k$ if necessary,
fix the character $\bar\psi_1$ on a $G_{\Q_p}$-stable line of $\bar\rho$.
Let $R^{\mathrm{ord}}$ be the universal ring representing pairs $(r,\psi_1)$,
where $r$ is a deformation of $\bar\rho$ with determinant $\chi$ and
$\psi_1:G_{\Q_p}\to B^\times$ lifts $\bar\psi_1$. Put $\psi_2=\chi|_{G_{\Q_p}}\psi_1^{-1}$.
We require $\tr(r|_{G_{\Q_p}})=\psi_1+\psi_2$ and
\begin{equation}\label{eq:ordinary-matrix-identity}
(r(\sigma)-\psi_1(\sigma)I)(r(\tau)-\psi_2(\tau)I)=0
\qquad(\sigma,\tau\in G_{\Q_p}).
\end{equation}
Then $R^{\mathrm{ord}}$ is a finite $\Lambda$-module via the weight map
of the universal character $\psi_1^{\mathrm{univ}}$.
\end{enumerate}
\end{theorem}
\begin{proof}
Parts~\textup{(i)} and~\textup{(ii)} are
\cite[Theorems~4.1.1(1) and~4.1.2(1)]{Zhang26}; see also
\cite[Theorems~5.1.1(1) and~6.1.1(1)]{Pan22}.
Part~\textup{(iii)} is \cite[Theorem~A.3.1]{Zhang26}.
\end{proof}

We recall the following result from the author's earlier work
\cite[proof of Theorem~3.9]{Luo26}.

\begin{lemma}\label{lem:fixed-specialization}
Let $G$ be a topologically finitely generated profinite group and
let $A$ be a pro-$p$ ring. Let $n\geq1$ be a fixed positive
integer and let $\rho:G\to\GL_n(A)$ be a continuous representation.
Suppose that, for any non-Archimedean local field $F$ with residue
field of characteristic $p$ and any continuous local homomorphism
$A\to\mathcal O_F$, the representation induced by $\rho$ over
$\mathcal O_F$ has finite image, where $\mathcal O_F$ denotes the
ring of integers of $F$. Then the image of $\rho$ contains a closed nilpotent subgroup of finite index.
\end{lemma}

\Needspace{7\baselineskip}
\section{Proof of Theorem 1.2 and Corollaries 1.3 and 1.4}
\label{sec:odd-completion}
Throughout this section, $p>2$.

\begin{lemma}\label{lem:quadratic-solvable}
Let $K$ be a number field, let $S$ be a finite set of primes of $K$,
and let $F$ be a non-Archimedean local field of residue characteristic $p$.
Let $\rho:G_{K,S}\to\GL_2(F)$ be a continuous representation with
finite $\rho(I_v)$ for every prime $v\mid p$ of $K$. If $\rho$ is not
absolutely irreducible, then it has finite image.
\end{lemma}
\begin{proof}
After a finite extension of $F$, the representation becomes upper
triangular. In particular, the image of $\rho$ is solvable. Then the assertion follows from Lemma~\ref{lem:arithmetic-finiteness}.
\end{proof}

\begin{lemma}\label{lem:local-finite}
	Let $F$ be a non-Archimedean local field, and let
	$\rho:G_{\Q_p}\to\GL_2(F)$ be a continuous absolutely irreducible
	representation. If $\rho(I_p)$ is finite and $\detm\rho$ has finite order,
	then $\rho$ has finite image.
\end{lemma}
\begin{proof}
	Choose a Frobenius lift $\phi\in G_{\Q_p}$ and put $J=\rho(I_p)$ and $M=\rho(\phi)$.
	Since $J$ is finite, some positive power $M^m$ centralizes $J$. Since $M^m$ commutes with $M$ and $\rho(G_{\Q_p})$ is topologically generated by $J$ and $M$, $M^{m}$ commutes with $\rho(G_{\Q_p})$. By Schur's lemma, we have $M^m=aI$ for $a\in F^\times$. Since $a^2=\detm(M)^m$ has
	finite order, so does $a$, and consequently $M$ has finite order. Since $J$ is finite, $\rho(G_{\Q_p})$ is also finite.
\end{proof}

\begin{proposition}\label{prop:char-p-finite-inertia}
Let $F$ be a local field of characteristic $p$, and let
$\rho:G_\Q\to\GL_2(F)$ be a continuous odd representation unramified outside finitely many primes. If $\rho(I_{p})$ is finite, then $\rho$ has finite image.
\end{proposition}
\begin{proof}
Choose a finite set $\Sigma$ of primes of $\Q$ containing $p$, $\infty$,
and all primes at which $\rho$ ramifies. Since $G_{\Q,\Sigma}$ is compact,
we may conjugate $\rho$ such that its image lies in $\GL_2(\mathcal O_F)$. By Lemma~\ref{lem:quadratic-solvable}, we may assume that $\rho$ is absolutely irreducible. Let $k$ be the residue field of $F$ and let
$\bar\rho$ denote the reduction of $\rho$.

Writing $q=|k|$, we identify $k$ with the subfield
$\{a\in\mathcal O_F:a^q=a\}$ via reduction, since every element of $k$
lifts uniquely to a root of $X^q-X$ in $\mathcal O_F$.
Let $s:W(k)\twoheadrightarrow k\hookrightarrow\mathcal O_F$ be reduction
followed by the inclusion of constants. For a $k^\times$-valued character
$\gamma$, write $\widetilde\gamma$ for its Teichm\"uller lift to $W(k)^\times$.
By Lemma~\ref{lem:arithmetic-finiteness}, $\detm\rho$ has finite image.
Since every finite-order unit of $F$ belongs to $k^\times$, we have
$\detm\rho=s\circ\widetilde{\detm\bar\rho}$.

Suppose that $\rho|_{G_{\Q_p}}$ is reducible over a finite extension $F'/F$.
If $\bar\rho$ is reducible over $\overline k$, write
$\bar\rho^{\ssem}=\bar\chi_1\oplus\bar\chi_2$. Their distinct values at complex conjugation
prevent $\Gal(\overline k/k)$ from exchanging them, so both are $k$-valued. After twisting $\rho$ by the finite-order
character $s\circ\widetilde{\bar\chi_1}^{-1}$, we may assume that
\[
 \bar\rho^{\ssem}=1\oplus\bar\chi,\qquad
 \detm\rho=s\circ\widetilde{\bar\chi},
\]
where $\bar\chi=\bar\chi_2\bar\chi_1^{-1}$ is odd. Let $k'$ be the residue field of $F'$, and choose a
$G_{\Q_p}$-stable line $V\subset (F')^2$. In the base change $L=\mathcal O_{F'}\otimes_{\mathcal O_F}\mathcal O_F^2$
of the stable lattice, $L\cap V$ is a rank-one summand. Write
$\alpha,\beta$ for the characters on $L\cap V$ and $L/(L\cap V)$,
respectively. Then $T=\tr\rho$ satisfies $T|_{G_{\Q_p}}=\alpha+\beta$.
The three cases of Theorem~\ref{thm:pan-ordinary} apply as follows.
\begin{itemize}
\item[\textup{(i)}] If $\bar\rho^{\ssem}=1\oplus\bar\chi$ and
$\bar\chi|_{G_{\Q_p}}\ne1$, then the reductions of $\alpha,\beta$ are $1,\bar\chi|_{G_{\Q_p}}$
in some order. Taking the unique character lifting $1$ as $\psi_1$ gives
the required ordering in the trace decomposition.
\item[\textup{(ii)}] If $\bar\rho^{\ssem}=1\oplus\bar\chi$ and
$\bar\chi|_{G_{\Q_p}}=1$, then both characters lift $1$. Take $\psi_1=\alpha$ and
$\psi_2=\beta$. In a basis adapted to $L\cap V$,
\eqref{eq:ordinary-matrix-identity} holds. Multiplying it by $\rho(h)$
on the right and taking traces gives \eqref{eq:pan-ordinary-identity}
for $(T,\alpha)$.
\item[\textup{(iii)}] If $\bar\rho$ is absolutely irreducible, then reducing
$L\cap V$ gives a $G_{\Q_p}$-stable line with character $\bar\alpha$.
Thus $(\rho,\alpha)$ has the prescribed residual character and satisfies
\eqref{eq:ordinary-matrix-identity} in the same adapted basis.
\end{itemize}

In each case the corresponding ordinary ring over $W(k')$ maps to
$\mathcal O_{F'}$, carrying the universal trace to $\tr\rho$.
Since $\rho(I_p)$ is finite, $\psi_1$ has finite image on inertia.
It is trivial on $U_1$ via local reciprocity because $U_1$ is pro-$p$
and $(F')^\times$ has no nontrivial $p$-power torsion.
Thus the weight map followed by specialization kills $\mathfrak a$.
Since $F'$ has characteristic $p$, the specialization factors through
$R^{\mathrm{ord}}/(p,\mathfrak a)R^{\mathrm{ord}}$, a finite $k'$-algebra
by Theorem~\ref{thm:pan-ordinary}.

Now suppose that $\rho|_{G_{\Q_p}}$ is absolutely irreducible. By Lemma \ref{lem:arithmetic-finiteness}, $\detm \rho$ has finite image and so does $\detm \rho|_{G_{\Q_p}} $. By Lemma~\ref{lem:local-finite}, $\rho|_{G_{\Q_p}}$ has finite image.
For each $g\in G_{\Q_p}$, the eigenvalues of $\rho(g)$ are roots of unity
and hence lie in $\overline{\F}_p$, since $F$ has characteristic $p$.
Their sum $\tr\rho(g)$ therefore belongs to the constant field
$F\cap\overline{\F}_p=k$. Thus $\tr\rho(g)=\iota(\tr\bar\rho(g))$,
where $\iota:k\hookrightarrow\mathcal O_F$ is the inclusion of constants.
Set $\bar T=\tr\bar\rho$ and $\chi=\widetilde{\detm\bar\rho}$.
The global trace $\tr\rho$ induces a map
$\theta:R_\Q^{\ps,\chi}\to\mathcal O_F$ from the universal
pseudodeformation ring of $\bar T$. By the local trace identity above, the
composite of $R_p^{\ps,\chi}\to R_\Q^{\ps,\chi}$ with $\theta$ is
$R_p^{\ps,\chi}\twoheadrightarrow k\xrightarrow{\iota}\mathcal O_F$.
Thus $\theta$ factors through $R_\Q^{\ps,\chi}/\mathfrak m_p R_\Q^{\ps,\chi}$,
where $\mathfrak m_p$ is the maximal ideal of $R_p^{\ps,\chi}$.
This quotient is a finite $k$-algebra by Theorem~\ref{thm:local-global-finite}.

Therefore, in either case the set $\{\tr\rho(g):g\in G_{\Q,\Sigma}\}$ is finite. Since $p>2$, our claim follows from Lemma~\ref{lem:finite-traces}.
\end{proof}

If $A$ is a pro-$p$ ring, then we say that a continuous representation of a profinite group $G$ on $A^2$ is \emph{reducible over $A$}
if $A^2$ contains a $G$-stable free $A$-submodule of rank one that is a direct summand as an $A$-module. 

\begin{proof}[Proof of Theorem~\ref{thm:finite-inertia-intro}]
Let $r$ be the specialization of $\rho$ under a continuous local map
$A\to\mathcal O_F$, where $F$ is a local field of residue characteristic
$p$. Then $r$ is odd, is unramified outside finitely many primes, and
has finite inertia image at $p$. If $F$ has characteristic $p$, then $r$ has finite image by Proposition~\ref{prop:char-p-finite-inertia}. Suppose that $F$ has characteristic zero. By
Lemma~\ref{lem:quadratic-solvable}, we may assume that $r$ is absolutely
irreducible. For $p>3$, $r$ has finite image by Theorem~\ref{thm:finite-inertia-char-zero}\textup{(i)}. For $p=3$, $r$ has finite image by Theorem~\ref{thm:finite-inertia-char-zero}\textup{(ii)}.

By Lemma~\ref{lem:arithmetic-finiteness}, the image
$\rho(G_\Q)$ is topologically finitely generated. Apply
Lemma~\ref{lem:fixed-specialization} to
$\rho(G_\Q)\hookrightarrow\GL_2(A)$ to see that $\rho(G_\Q)$ is virtually nilpotent. By Lemma~\ref{lem:arithmetic-finiteness}, $\rho(G_\Q)$ is finite.
\end{proof}

\begin{proof}[Proof of Corollary~\ref{thm:combined-intro}]
Inflate the representation to $G_\Q$. Since $p\notin S$, its inertia
image at $p$ is trivial. For $p>3$, apply
Theorem~\ref{thm:finite-inertia-intro}\textup{(i)}. For $p=3$, put
$\tau=(\bar\rho|_{G_{\Q_3}})^{\ssem}$. Since $\tau$ is trivial on
$I_3$ and $\omega$ is nontrivial on $I_3$, we have
$\tau\not\simeq\tau\otimes\omega$. Theorem~\ref{thm:finite-inertia-intro}\textup{(iv)}
then applies.
\end{proof}

\begin{lemma}\label{lem:local-reducibility}
	Let $F/\Q_3$ be a finite extension and let
	$\rho:G_{\Q_3}\to\GL_2(F)$ be a continuous representation.
	If $\rho(I_3)$ is central in $\rho(G_{\Q_3})$, then $\rho$ is reducible
	after a finite extension of $F$.
\end{lemma}
\begin{proof}
	The group $G_{\Q_3}$ is topologically generated by $I_3$ and a
	Frobenius lift. Since $\rho(I_3)$ is central in the image, the image is
	abelian. If it consists of scalars, every line is stable. Otherwise,
	choose a nonscalar element of the image and adjoin one of its eigenvalues
	to $F$. The corresponding eigenspace is a line preserved by the whole image.
\end{proof}

\begin{proof}[Proof of Corollary~\ref{thm:descent-intro}]
Put $\Gamma=\Gal(K_S/\Q)$. The stability of $S$ makes
$K_S/\Q$ Galois. Inflate the given
representation along $G_\Q\twoheadrightarrow\Gamma$. It is
unramified outside the rational primes ramified in $K/\Q$, the
primes below $S$, and $\infty$. Since $K_S/K$ is unramified
above $p$, this inflated representation kills $I_p\cap G_K$.
Consequently
\begin{equation}\label{eq:descent-inertia}
\rho(I_p)\text{ is finite, of order dividing }e(v/p),
\qquad v\mid p.
\end{equation}
If $p>3$ or $A$ has positive characteristic, apply
Theorem~\ref{thm:finite-inertia-intro}\textup{(i)} or~\textup{(ii)}.
Under condition~\textup{(iii)}, apply
Theorem~\ref{thm:finite-inertia-intro}\textup{(iii)}.
Under condition~\textup{(iv)}, apply
Theorem~\ref{thm:finite-inertia-intro}\textup{(iv)} directly.

For~\textup{(v)}, we verify the weaker specialization condition
used in the proof of Theorem~\ref{thm:finite-inertia-intro}.
Let $r$ be a specialization into a finite extension of $\Q_3$.
If $K_v/\Q_3$ is abelian, then for $d\in G_{\Q_3}$ and
$i\in I_3$, the commutator $[d,i]$ lies in $I_3\cap G_K$
and is killed by $r$. Thus inertia is central in the local
image, and Lemma~\ref{lem:local-reducibility} gives geometric
local reducibility.

Under~\textup{(vi)}, inertia is finite cyclic of order
$n=\#\rho(I_3)$ prime to $3$. Reduction is injective on this
image because the congruence kernel is pro-$3$.
For a lift $t$ of a topological generator of tame inertia,
$M=\bar\rho(t)$ is semisimple of order $n$.
Arithmetic Frobenius conjugates $M$ to $M^3$, so cubing permutes
its eigenvalues. If it fixes both, then they lie in $\{1,-1\}$ and
$n\mid2$. Otherwise they are $\alpha,\alpha^3$, where
$\alpha^8=1$ and $\alpha^2\ne1$, so $n=4$ or $8$.
The hypothesis excludes $n=4$.

If $n\leq2$, then inertia is central, and the preceding specialization argument applies.
If $n=8$, then $\alpha$ has order $8$ and $\omega(t)=-1$.
The eigenvalues of $\tau(t)$ are $\alpha,\alpha^3$, whereas
those of $(\tau\otimes\omega)(t)$ are $\alpha^5,\alpha^7$.
These pairs are disjoint, so
$\tau\not\simeq\tau\otimes\omega$. Thus~\textup{(iv)} applies.
\end{proof}

	\section{Representations of pro-$p$ Galois groups over imaginary quadratic fields}
\label{sec:quadratic}
Throughout this section, $p>2$, $A$ is a pro-$p$ ring, and
$K/\Q$ is imaginary quadratic. Fix a complex conjugation
$c\in G_\Q$.

\subsection{Conjugate self-duality and two generators}
We first work over local fields and then pass to general pro-$p$ rings by specialization.
Write $\varepsilon_K:G_\Q\to\{\pm1\}$ for the quadratic character associated with $K/\Q$.

\begin{proposition}\label{prop:quadratic-descent}
Let $p>2$, let $K/\Q$ be imaginary quadratic, and let $S$ be a
finite set of primes of $K$ containing $S_\infty(K)$ and no prime
above $p$. Let $F$ be a non-Archimedean local field of residue characteristic $p$. Suppose that
\[
 \rho:G_K\longrightarrow\GL_2(F)
\]
is a continuous absolutely irreducible representation, which is unramified outside $S$. Suppose that the image of $\rho$ is a pro-$p$ group and there exists $C\in\GL_2(F)$ satisfying
\[
 C^2=I,\qquad \detm C=-1,\qquad
 C\rho(g)C^{-1}=\rho(cgc^{-1})\quad\forall g\in G_K.
\]
Then $\rho$ has finite image.
\end{proposition}

\begin{proof}
The actual involution $c$ splits $G_{\Q}\to\Gal(K/\Q)$, and hence
\[
 G_{\Q}=G_K\rtimes\langle c\rangle.
\]
Every element of $G_{\Q}$ has a unique expression $gc^\epsilon$, with $g\in G_K$ and $\epsilon\in\{0,1\}$. Define
\[
 \widetilde\rho(gc^\epsilon):=\rho(g)C^\epsilon.
\]
The displayed relations ensure multiplicativity. Since the two cosets of $G_K$ are open and closed, $\widetilde\rho$ is continuous. Since the restriction to $G_K$ is absolutely irreducible, $\widetilde\rho$ is absolutely irreducible. It is odd because $\det\widetilde\rho(c)=\detm C=-1$.

Note that the representation $\widetilde\rho$ is unramified outside the rational
primes below $S$ and those dividing $\operatorname{disc}(K)$. Moreover, $I_p\cap G_K$ is the inertia group at a prime of $K$ above $p$
and is killed by $\rho$, since $S$ contains no prime above $p$. Hence
\[
 |\widetilde\rho(I_p)|\leq2.
\]
The compact image of $\widetilde\rho$ preserves an $\mathcal O_F$-lattice
$T$ in $F^2$. If $p>3$ or $F$ has positive characteristic, then $\widetilde{\rho}$ has finite image by Theorem~\ref{thm:finite-inertia-intro}\textup{(i)} or~\textup{(ii)},
respectively.

Suppose that $p=3$ and $F$ has characteristic zero. Since
$\widetilde\rho(I_3)$ is normal in $\widetilde\rho(G_{\Q_3})$ and has order at most two, it is central. By Lemma~\ref{lem:local-reducibility}, $\widetilde{\rho}|_{G_{\Q_3}}$ is reducible after a finite extension of $F$. By Theorem~\ref{thm:finite-inertia-intro}\textup{(iii)}, $\widetilde{\rho}$ has finite image. In either case, $\rho$ has finite image. This completes the proof.
\end{proof}

\begin{proposition}
\label{prop:quadratic-field-csd}
Let $S$ be a finite set of primes of $K$ containing $S_\infty(K)$
and no prime above $p$. Suppose that at most one prime $\ell\in T_K(S)$ is congruent
to $1$ modulo $p$, and let $F$ be a non-Archimedean local field of residue characteristic $p$.  Then every continuous conjugate self-dual representation $\rho:G_{K,S}(p)\to\GL_2(F)$ has finite image.
\end{proposition}
\begin{proof}
By Lemma~\ref{lem:quadratic-solvable}, we may assume that $\rho$ is
absolutely irreducible.  After a finite extension of $F$, the
conjugate self-duality isomorphism is defined over $F$.  Thus
$\ker \rho_K$ is stable under $c$, because a representation and its
dual have the same kernel. It follows that $\rho_K$ is unramified outside
\[
 S_0:=S_\infty(K)\cup\{v:\ v\mid\ell\text{ for some }\ell\in T_K(S)\}.
\]

The determinant $\delta=\detm \rho$ has finite image of odd order $N$,
because it factors through the finite abelian $p$-group
$G_{K,S}(p)^{\ab}$.  Conjugate self-duality gives
$\delta_K^c=\delta_K^{-1}$.  Put $\eta=\delta^{(N+1)/2}$ and
$\rho_0=\eta^{-1}\rho$.  Then $\eta^2=\delta$,
$\eta_K^c=\eta_K^{-1}$, and $\detm \rho_0=1$. Moreover, we have
$\ker \rho_K\subseteq\ker\eta_K$, and hence the inflation $\rho_{0,K}$ is
unramified outside $S_0$.  The canonical self-duality of a
two-dimensional determinant-one representation gives
\[
 \rho_{0,K}^c\simeq \rho_{0,K}^\vee\simeq \rho_{0,K}.
\]
After taking a finite extension of the coefficient field $F$, still
denoted by $F$, we can choose $B\in\GL_2(F)$ such that $B\rho_{0,K}(g)B^{-1}=\rho_{0,K}(cgc^{-1})$. Since $c^{2}=1$, Schur's lemma implies that $B^2=\lambda I$ for some $\lambda\in F^\times$. Enlarging $F$ to contain $\sqrt{\lambda}$, put $C:=\frac{B}{\sqrt{\lambda}}$. Then $C\in\GL_2(F)$ satisfies $C^2=I$ and
$C\rho_{0,K}(g)C^{-1}=\rho_{0,K}(cgc^{-1})$ for every $g\in G_K$.

If $C$ is not scalar, then its eigenvalues are $1$ and $-1$. In particular, we have $\detm C=-1$. Applying Proposition~\ref{prop:quadratic-descent}
to $\rho_{0,K}$ with ramification set $S_0$, we obtain that $\rho_0$ has finite image.

Suppose that $C$ is scalar.  Then $\rho_{0,K}(cgc^{-1})=\rho_{0,K}(g)$,
and the formula
\[
 \widetilde \rho_0(gc^e)=\rho_{0,K}(g)
 \qquad(g\in G_K,\ e\in\{0,1\})
\]
defines a continuous representation of $G_\Q$ whose image is a pro-$p$ group. Put $T:=T_K(S)$.  For $\ell\notin T$, the restriction of
$\rho_{0,K}$ to inertia at every prime of $K$ above $\ell$ is
trivial.  Therefore $|\widetilde \rho_0(I_\ell)|\leq2$.  As $p$ is
odd and the image is pro-$p$, this inertia image is trivial. Hence $\widetilde \rho_0$ factors through $G_{\Q,T\cup\{\infty\}}(p)$. By class field theory, the abelianization of $G_{\Q,T\cup\{\infty\}}(p)$ is finite and cyclic under our assumption. The Burnside basis theorem then shows that the group itself is finite cyclic. Thus $\rho_0$ has finite image also in the scalar case.
Finally, since $\eta$ has finite image, $\rho$ has finite image. This finishes the proof of our assertion.
\end{proof}

Write $\Phi(G)=\overline{G^p[G,G]}$ for the Frattini subgroup of a pro-$p$ group $G$, and $d(G)=\dim_{\F_p}G/\Phi(G)$ for its generator rank. The following elementary construction provides the odd involution needed for descent.

\begin{lemma}\label{lem:two-generator-involution}
	Let $p$ be odd and let $F$ be a non-Archimedean local field of residue characteristic $p$. Let $G$ be a pro-$p$ group with $\dim_{\F_p}G/\Phi(G)=2$ and an involution $\alpha$ acting by $-1$ on $G/\Phi(G)$. If $\rho:G\to\SL_2(F)$ is a continuous absolutely irreducible representation, then, after a finite extension of $F$, there is a matrix $C\in\GL_2(F)$ such that
	\[
	C^2=I,\qquad\det C=-1,\qquad C\rho(g)C^{-1}=\rho(\alpha(g))
	\quad(\forall g\in G).
	\]
\end{lemma}

\begin{proof}
	Lift a basis of $G/\Phi(G)$ to $y_1,y_2\in G$ and put $x_i=y_i\alpha(y_i)^{-1}$. Then we have $\alpha(x_i)=x_i^{-1}$. Since their images in the Frattini quotient are twice the chosen basis vectors, we see that $x_1,x_2$ topologically generate $G$ by the Burnside basis theorem.
	
	Set $X=\rho(x_1)$, $Y=\rho(x_2)$ and $B=XY-YX$. By the Cayley--Hamilton theorem, we have
	\[
	BX+XB=\tr(X)B,\qquad BY+YB=\tr(Y)B.
	\]
	Since $X,Y\in\SL_2(F)$, these identities are equivalent to
	\[
	BX=X^{-1}B,\qquad BY=Y^{-1}B.
	\]
	If $B$ were nonzero and singular, then its kernel would be a common invariant line for $X,Y$. If $B=0$, then the commuting pair would again have a common invariant line over an algebraic closure. Since $x_1,x_2$ topologically generate $G$, both conclusions contradict the absolute irreducibility of $\rho$. Thus $B$ is invertible. Moreover $\tr(B)=0$, whence $B^2=-\det(B)I$.
	
	We choose $s$ with $s^2=-\det B$ and put $C=s^{-1}B$. Then we have $C^2=I$, $\det C=-1$, and $C$ inverts both $X$ and $Y$ by conjugation. Thus $C\rho(g)C^{-1}=\rho(\alpha(g))$ on the generators and, by continuity, on all of $G$.
\end{proof}

\begin{corollary}
\label{cor:quadratic-field-rank2}
Let $F$ be a non-Archimedean local field of residue characteristic $p$.  If $d_p\Cl(K)=2$, then every continuous representation $\rho:G_{K,\varnothing}(p)\to\GL_2(F)$ has finite image.
\end{corollary}
\begin{proof}
Put $G=G_{K,\varnothing}(p)$, and let $\alpha$ be the involution
induced by $c$. By class field theory,
$G/\Phi(G)\simeq\Cl(K)/p\Cl(K)$, and $\alpha$ acts by $-1$
on this two-dimensional $\F_p$-vector space.
By Lemma~\ref{lem:quadratic-solvable}, we may assume that $\rho$ is absolutely
irreducible.  As in Proposition~\ref{prop:quadratic-field-csd},
twist by the finite-image square root of $\detm \rho$ to obtain
$\rho_0:G_{K,\varnothing}(p)\to\SL_2(F)$.
Lemma~\ref{lem:two-generator-involution} supplies an odd involution
after a finite scalar extension. Applying
Proposition~\ref{prop:quadratic-descent} to the inflation of $\rho_0$
to $G_K$, with $S=S_\infty(K)$, gives finite image, and hence
$\rho$ has finite image.
\end{proof}

\begin{proof}[Proof of Theorem~\ref{thm:csd-intro}]
For general $A$, conjugate self-duality specializes to every
		$\mathcal O_F$. By Proposition~\ref{prop:quadratic-field-csd}, every such specialization has finite image. The image $H=\rho(G_{K,S}(p))$ is topologically finitely generated
		by Lemma~\ref{lem:arithmetic-finiteness}. Applying
		Lemma~\ref{lem:fixed-specialization} to $H\hookrightarrow\GL_2(A)$
		shows that $H$ is virtually nilpotent, hence finite by
		Lemma~\ref{lem:arithmetic-finiteness}. 
\end{proof}

\begin{proof}[Proof of Theorem~\ref{thm:rank2-intro}]
If $d_p\Cl(K)\leq1$, then $G_{K,\varnothing}(p)$ is procyclic, and hence is finite by class field theory. Suppose that $d_p\Cl(K)=2$.
By Corollary~\ref{cor:quadratic-field-rank2}, every local-field specialization has finite image. The image $H=\rho(G_{K,\varnothing}(p))$ is topologically finitely
	    generated by Lemma~\ref{lem:arithmetic-finiteness}. Apply
	    Lemma~\ref{lem:fixed-specialization} to $H\hookrightarrow\GL_2(A)$,
	    and then Lemma~\ref{lem:arithmetic-finiteness}, to conclude that $H$
	    is finite.
\end{proof}

	\subsection{Faithful representations of the whole tower group}
	
	\begin{proof}[Proof of Theorem~\ref{thm:faithful-tower}]
		 If $d_p\Cl(K)\leq2$, then Theorem~\ref{thm:rank2-intro} makes the given representation finite, and faithfulness makes $G$ finite.
		
		For $d_p\Cl(K)\geq3$, the argument of Koch and Venkov~\cite{KV74} shows that $G_{K,\varnothing}(p)$ is a Golod--Shafarevich pro-$p$ group
		(see also \cite[proof of Corollary~2.8]{McL08}). By Zelmanov's
		theorem~\cite{Zelmanov00}, it contains a closed nonabelian free pro-$p$
		subgroup $F$. Zubkov's theorem~\cite{Zubkov87} rules out a faithful
		continuous representation of such a group by two-by-two matrices over
		a commutative profinite ring. Since $A$ is a profinite ring, $d_p\Cl(K)\geq3$ is impossible under the embedding hypothesis. This completes
		the proof.
	\end{proof}

	\section{McLeman's \texorpdfstring{$(3,3)$}{(3,3)}-conjecture}
	\label{sec:mcleman}
	
	For a finitely generated pro-$p$ group $F$, let $D_n(F)$ be its Zassenhaus filtration:
	\[
	D_n(F)=\{g\in F:g-1\in I_F^n\},
	\]
	where $I_F$ is the augmentation ideal of $\F_p[[F]]$.
	
	\begin{definition}\cite[Section~2]{McL08}.\label{def:33}
		A pro-$p$ group has \emph{Zassenhaus type $(3,3)$} if it admits a minimal presentation $G=F/R$ with $F$ free pro-$p$ on two generators, such that $R$ is normally generated by two elements of $D_3(F)$ having linearly independent images in $D_3(F)/D_4(F)$.
	\end{definition}
	
	Recall that a \emph{strong Schur $\sigma$-group} is an FAb pro-$p$ group $G$ equipped with an involution $\alpha$ such that
	\[
	\dim_{\F_p}H^2(G,\F_p)\leq\dim_{\F_p}H^1(G,\F_p)<\infty,
	\]
	and $\alpha$ acts by $-1$ on both cohomology spaces. Here FAb means that every open subgroup has finite abelianization. We use the following embedding consequence of Pink's structural theorem
	\cite[Proposition~8.1]{Pink25}.
	
	\begin{theorem}\label{thm:pink}
		Let $p>3$. Then every infinite strong Schur $\sigma$-group of Zassenhaus type $(3,3)$ is isomorphic to an open subgroup of $\GG(\Q_p)$ for a $\Q_p$-form $\GG$ of $\PGL_2$.
	\end{theorem}
	
	For imaginary quadratic $K$, the group $G_{K,\emptyset}(p)$, with the action induced by complex conjugation, is a strong Schur $\sigma$-group; see \cite[Section~9]{Pink25}.
	
	The projective embedding needed here also follows from the earlier classification of Andozhskii.
	
	\begin{theorem}\label{thm:andozhskii}
		Let $p>3$ be a prime and let $G$ be an infinite pro-$p$ group with $d(G)=\dim_{\F_p}H^2(G,\F_p)=2$, finite abelianization, and Zassenhaus type $(3,3)$. Then there exist a finite extension $E/\Q_p$ and a continuous injective homomorphism
		\[
		j:G\hookrightarrow\PGL_2(E).
		\]
	\end{theorem}
	
	\begin{proof}
		In the terminology of \cite[p.~665]{Andozhskii75}, \emph{closed} refers to equal generator and relation ranks, and \emph{commutator-finite} to finite abelianization. For $p\geq5$, the two-dimensional degree-three quotient of the free pro-$p$ group on $x,y$ is spanned by $[x,[x,y]]$ and $[y,[x,y]]$. Independence of the two cubic initial forms is exactly the invertibility condition in \cite[Proposition~4.2]{Andozhskii75}.
		
		By \cite[Theorem~4.1]{Andozhskii75}, $G$ is $p$-saturated of rank three. The classification in \cite[Theorem~6.2 and (5.10)]{Andozhskii75} identifies $G$ itself with a type-II orthogonal congruence subgroup. The corresponding ternary quadratic form $q$ has diagonal coefficients
		\[
		\lambda\bigl(\mu_1p^{n_2+n_3},\mu_2p^{n_1+n_3},\mu_3p^{n_1+n_2}\bigr),
		\]
		with $\lambda\in\Z_p\setminus\{0\}$, $\mu_i\in\Z_p^\times$, and in the two-generator case $n_1\geq n_2>0$, $n_3=0$. Thus $q$ is nondegenerate over $\Q_p$. Write $q(v)=v^{\mathsf T}Bv$, where $B$ is the diagonal matrix with these coefficients. For any field extension $L/\Q_p$, the notation
		\[
		\operatorname{SO}(q)(L)
		=\{g\in\GL_3(L):g^{\mathsf T}Bg=B,\ \detm g=1\}
		\]
		denotes the special orthogonal group of $q$ over $L$, consisting of the determinant-one $L$-linear transformations of $L^3$ that preserve the scalar extension of $q$ to $L$. We equip $\operatorname{SO}(q)(\Q_p)$ with its usual $p$-adic topology. Since $G$ is pro-$p$ and $p$ is odd, the determinant character of this orthogonal representation, with values in $\{\pm1\}$, is trivial. The classification above identifies $G$, as a topological group, with an open subgroup of $\operatorname{SO}(q)(\Q_p)$.
		
		Choose a finite extension $E/\Q_p$ over which $q$ is split. The adjoint action on $\mathfrak{sl}_2(E)$ identifies $\PGL_{2,E}$ with the split special orthogonal group of the ternary quadratic form $X\mapsto-\detm X$. Since a split nondegenerate ternary quadratic form is similar to this form, we obtain a continuous injective homomorphism
		\[
		G\hookrightarrow\operatorname{SO}(q)(\Q_p)
		\hookrightarrow\operatorname{SO}(q)(E)\simeq\PGL_2(E),
		\]
		as required.
	\end{proof}
	
	The next lemma lifts this projective embedding to a continuous injective homomorphism into $\SL_2(E)$.
	
	\begin{lemma}\label{lem:lift}
		Let $p$ be odd, let $E/\Q_p$ be a finite extension and let $H$ be a compact pro-$p$ subgroup of $\PGL_2(E)$. Then the inclusion of $H$ lifts to a continuous injective homomorphism $s:H\to\SL_2(E)$.
	\end{lemma}
	
	\begin{proof}
		Note that the determinant map $\PGL_2(E)\to E^\times/(E^\times)^2$ is continuous with finite $2$-group target, and hence it is trivial on $H$. It follows that $H$ lies in the image of the natural map $q:\SL_2(E)\to\PGL_2(E)$. Then $\widetilde H:=q^{-1}(H)$ is compact. We therefore have an exact sequence of profinite groups
		\[
		1\longrightarrow\{\pm I\}\longrightarrow\widetilde H
		\xrightarrow{q}H\longrightarrow1.
		\]
		As $p$ is odd, this sequence splits by the profinite Schur--Zassenhaus theorem \cite[Theorem~2.3.15]{RZ10}. Then a continuous section, followed by the inclusion $\widetilde H\hookrightarrow\SL_2(E)$, gives the required lift.
	\end{proof}
	
	\begin{proof}[Proof of Theorem~\ref{thm:mcleman-pgt3}]
		Suppose that $G=G_{K,\emptyset}(p)$ is infinite. Then it has two generators, two relations, and finite abelianization. By Theorem~\ref{thm:andozhskii} or Theorem \ref{thm:pink}, we have a faithful map $j:G\hookrightarrow\PGL_2(E)$ after a finite extension of $\Q_p$. Since its image is compact and pro-$p$, we have a faithful lift to $\SL_2(E)$ by Lemma~\ref{lem:lift}. By Theorem~\ref{thm:rank2-intro}, $G$ is finite, contradicting the infinitude of $G$.
	\end{proof}
	
	Note that the restriction $p>3$ in the proof of Theorem~\ref{thm:mcleman-pgt3} only occurs in the structural results of Pink and Andozhskii. Ahlqvist and Pink provide the required embedding in ten cases at $p=3$ in \cite[Proposition~9.2]{AP26}.
	
	\begin{proof}[Proof of Theorem~\ref{thm:mcleman-p3}]
		If $G_{K,\emptyset}(3)$ were infinite, then \cite[Proposition~9.2(a)]{AP26} would embed it as an open subgroup of a $\Q_3$-form of $\PGL_2$. After a finite extension, Lemma~\ref{lem:lift} would give a faithful two-dimensional representation, exactly as in the proof of Theorem~\ref{thm:mcleman-pgt3}. Then Theorem~\ref{thm:rank2-intro} implies that its image is finite, a contradiction.
	\end{proof}
	
	\begin{remark}\label{rem:p3-remaining}
		The classification in \cite[Section~8]{AP26} gives thirteen possible fourth Zassenhaus quotients of type $(3,3)$ at $p=3$. The remaining cases $[243,3]$, $[243,9]$, and $[243,13]$ are still open.
	\end{remark}

		\section*{Acknowledgments}
	This paper grew out of a discussion with Xinyao Zhang. After Zhang shared his manuscript \cite{Zhang26}, the author immediately recognized that \cite[Theorem~5.4.1]{Zhang26} could be used to prove the sufficiency direction of McLeman's $(3,3)$-conjecture for $p>3$. Having completed that proof, the author soon realized that Zhang's results could also be used to establish the two-dimensional odd case of Boston's unramified Fontaine--Mazur conjecture over $\Q$, as stated in Corollary~\ref{thm:combined-intro}. The author is especially grateful to Zhang for many helpful discussions. The author also thanks Prof.~Eric Ahlqvist for discussions on McLeman's $(3,3)$-conjecture, and Prof.~Christian Maire for comments on the arithmetic scope of the descent argument and for pointing out that the required $\GL_2$-embedding for $p>3$ can also be obtained from Andozhskii's paper \cite{Andozhskii75}.
	
	\begingroup
	\renewcommand{\addcontentsline}[3]{}

	\section*{AI use statement}
	During the preparation of this manuscript, the author utilized ChatGPT~6 to assist with language refinement, text proofreading, and verifying the logical consistency of mathematical proofs.
	\endgroup
	
	\providecommand{\bysame}{\leavevmode\hbox to3em{\hrulefill}\thinspace}
	\providecommand{\MR}{\relax\ifhmode\unskip\space\fi MR }
	\providecommand{\MRhref}[2]{%
		\href{http://www.ams.org/mathscinet-getitem?mr=#1}{#2}
	}
	\providecommand{\href}[2]{#2}


\begin{thebibliography}{10}
		
		\bibitem{AC25}
		E.~Ahlqvist and M.~Carlson, \emph{Massey products in the \'etale cohomology of
			number fields}, J. Reine Angew. Math. \textbf{823} (2025), 61--112.
		
		\bibitem{AP26}
		E.~Ahlqvist and R.~Pink, \emph{Schur {$\sigma$}-groups of type {$(3,3)$} for
			{$p=3$}}, 2026, Preprint,
		\href{https://arxiv.org/abs/2602.09889v1}{arXiv:2602.09889v1}.
		
		\bibitem{AllenCalegari14}
		P.~B. Allen and F.~Calegari, \emph{Finiteness of unramified deformation rings},
		Algebra Number Theory \textbf{8} (2014), no.~9, 2263--2272.
		
		\bibitem{Andozhskii75}
		I.~V. Andozhskii, \emph{On some classes of closed pro-{$p$}-groups}, Math.
		USSR-Izv. \textbf{9} (1975), no.~4, 663--691, Russian original: Izv. Akad.
		Nauk SSSR Ser. Mat. 39 (1975), no. 4, 707--738.
		
		\bibitem{Boston99}
		N.~Boston, \emph{Some cases of the {Fontaine--Mazur} conjecture, {II}}, J.
		Number Theory \textbf{75} (1999), no.~2, 161--169.
		
		\bibitem{Calegari18}
		F.~Calegari, \emph{Non-minimal modularity lifting in weight one}, J. Reine
		Angew. Math. \textbf{740} (2018), 41--62.
		
		\bibitem{CalegariGeraghty18}
		F.~Calegari and D.~Geraghty, \emph{Modularity lifting beyond the
			{Taylor--Wiles} method}, Invent. Math. \textbf{211} (2018), no.~1, 297--433.
		
		\bibitem{FM95}
		J.-M. Fontaine and B.~Mazur, \emph{Geometric {Galois} representations},
		Elliptic curves, modular forms, \& {Fermat}'s last theorem ({Hong Kong},
		1993) (J.~Coates and S.-T. Yau, eds.), Ser. Number Theory, vol.~I,
		International Press, Cambridge, MA, 1995, pp.~41--78.
		
		\bibitem{GS64}
		E.~S. Golod and I.~R. Shafarevich, \emph{On the class field tower}, Izv. Akad.
		Nauk SSSR Ser. Mat. \textbf{28} (1964), no.~2, 261--272, Russian.
		
		\bibitem{Luo23Thesis}
		Y.~Luo, \emph{On the unramified {Fontaine--Mazur} conjecture and its
			generalizations}, Ph.D. thesis, Humboldt-Universit{\"a}t zu Berlin, 2023.
		
		\bibitem{Luo26}
		\bysame, \emph{Remarks on the {Boston's} unramified {Fontaine--Mazur}
			conjecture}, J. Number Theory \textbf{281} (2026), 96--109.
		
		\bibitem{McL08}
		C.~McLeman, \emph{{$p$}-tower groups over quadratic imaginary number fields},
		Ann. Sci. Math. Qu\'ebec \textbf{32} (2008), no.~2, 199--209.
		
		\bibitem{Pan22}
		L.~Pan, \emph{The {Fontaine--Mazur} conjecture in the residually reducible
			case}, J. Amer. Math. Soc. \textbf{35} (2022), no.~4, 1031--1169.
		
		\bibitem{PilloniStroh16}
		V.~Pilloni and B.~Stroh, \emph{Surconvergence, ramification et modularit{\'e}},
		Ast{\'e}risque \textbf{382} (2016), 195--266.
		
		\bibitem{Pink25}
		R.~Pink, \emph{Schur {$\sigma$}-groups of type {$(3,3)$}}, 2025, Preprint,
		\href{https://arxiv.org/abs/2505.05580v2}{arXiv:2505.05580v2}.
		
		\bibitem{RZ10}
		L.~Ribes and P.~Zalesskii, \emph{Profinite groups}, 2nd ed., Ergebnisse der
		Mathematik und ihrer Grenzgebiete. 3. Folge, vol.~40, Springer-Verlag,
		Berlin, 2010.
		
		\bibitem{Sha66}
		I.~R. Shafarevich, \emph{Extensions with given points of ramification}, Amer.
		Math. Soc. Transl. (2) \textbf{59} (1966), 128--149.
		
		\bibitem{KV74}
		B.~B. Venkov and H.~Koch, \emph{The {$p$}-tower of class fields for an
			imaginary quadratic field}, Zap. Nau\v{c}n. Sem. Leningrad. Otdel. Mat. Inst.
		Steklov. (LOMI) \textbf{46} (1974), 5--13, 140, Russian.
		
		\bibitem{Zelmanov00}
		E.~Zelmanov, \emph{On groups satisfying the {Golod--Shafarevich} condition},
		New horizons in pro-{$p$} groups (M.~du~Sautoy, D.~Segal, and A.~Shalev,
		eds.), Progress in Mathematics, vol. 184, Birkh{\"a}user, Boston, 2000,
		pp.~223--232.
		
		\bibitem{Zhang26}
		X.~Zhang, \emph{Modular points and dimensions of {Eisenstein} deformation
			spaces}, preprint, \href{https://arxiv.org/abs/2512.21249v4}{arXiv:2512.21249v4}
		[math.NT] (2026).

		\bibitem{Zubkov87}
		A.~N. Zubkov, \emph{Non-abelian free pro-{$p$}-groups cannot be represented by
			2-by-2 matrices}, Siberian Math. J. \textbf{28} (1987), no.~5, 742--747.
		
	\enlargethispage{3\baselineskip}
	\end{thebibliography}
\end{document}